\documentclass[12 pt]{article}

\usepackage{amsmath, amsthm, latexsym}
\usepackage[all]{xy}
\usepackage{amsfonts}
\usepackage{amssymb}
\usepackage{color, soul, mathabx}
\usepackage{graphicx}
\usepackage{authblk}
\usepackage[utf8]{inputenc}
\usepackage[hidelinks]{hyperref}
\usepackage{orcidlink}

\newtheorem{theorem}{Theorem}[section]

\newtheorem{prop}[theorem]{Proposition}

\theoremstyle{definition}
\newtheorem{dfn}[theorem]{Definition}

\theoremstyle{remark}
\newtheorem{remark}[theorem]{Remark}

\newcommand{\conjC}{\mathbb{C}}
\newcommand{\conjR}{\mathbb{R}}
\newcommand{\conjN}{\mathbb{N}}
\newcommand{\conjK}{\mathbb{K}}
\DeclareMathOperator{\na}{NA }
\DeclareMathOperator{\interior}{int }

\title{Stability results for norm attaining operators with range in vector-valued continuous function spaces}

\author{$^1$Thiago Grando \orcidlink{0000-0002-7590-7358}, 
$^2$Elisa R. Santos \orcidlink{0000-0002-4575-3001}}

\affil{$^1$ \small Department of Mathematics, Midwestern Paraná State University, \\ Guarapuava, Paraná, Brazil.
\texttt{tgrando@unicentro.br}}

\affil{$^2$Institute of Mathematics and Statistics, Federal University of Uberl\^andia, \\ Uberlândia, Minas Gerais, Brazil.
\texttt{elisars@ufu.br}}

\date{ }

\begin{document}

\maketitle

\begin{abstract}
In this paper, we prove stability results for the denseness of norm-attaining operators (resp. compact operators), the Bishop-Phelps-Bollobás property for operators (resp. compact operators), and the generalized approximate hyperplane series property when the range space is a space of vector-valued continuous functions vanishing at infinity. We also discuss extensions of these results to spaces of vector-valued continuous functions on compact Hausdorff spaces and of bound- ed continuous functions on completely regular Hausdorff spaces.
\end{abstract}

\

\textsc{Keywords:} norm attaining operators, Bishop-Phelps-Bollobás property, generalized approximate hyperplane series property, compact operators, spaces of continuous functions.

\

\textsc{MSC 2020:} Primary 46B20; Secondary 46B25, 46E40.
\let\thefootnote\relax\footnotetext{The second author was partially supported by FAPEMIG RED-00133-21 and FAPEMIG APQ-01853-23.}

\section{Introduction} 

Let $X$ and $Y$ be real or complex Banach spaces. We denote by $B_X$ the closed unit ball, by $S_X$ the unit sphere, and by $X^*$ the topological dual of $X$. We write $\mathcal{L}(X,Y)$ for the Banach space of all bounded linear operators from $X$ into $Y$, $\na(X,Y)$ for the set of all norm attaining operators from $X$ into $Y$, and $\mathcal{K}(X,Y)$ for the space of all compact operators from $X$ into $Y$.

The classical Bishop–Phelps theorem \cite{BP1961} states that $\na(X,\conjK)$ is norm dense in $X^*$, where $\conjK = \conjR$ or $\conjC$. Bollobás \cite{B1970} later obtained a quantitative version of this result,
now known as the Bishop-Phelps-Bollobás theorem.

In 2008, Acosta et al. \cite{AAGM08} initiated the study of this theorem in the context of operators between Banach spaces $X$ and $Y$, and introduced the Bishop–Phelps–Bollobás property for operators.

\begin{dfn}\label{def-BPBP-operadores}
Let $X, Y$ be Banach spaces. We say that $(X,Y)$ has the  \textit{Bishop-Phelps-Bollobás property for operators} (BPBp for short) if, given $\varepsilon > 0$, there is $\eta(\varepsilon) > 0$ such that for every $T \in S_{\mathcal{L}(X,Y)}$ and $x_0 \in S_{X}$ satisfying $\Vert  T x_0 \Vert > 1- \eta(\varepsilon),$ there exist an operator $S \in S_{{\mathcal{L}(X,Y)}}$ and a point $u_0 \in S_{X}$ such that
$$
\| Su_0 \| = 1, \quad \Vert u_0 - x_0 \Vert < \varepsilon, \quad \text{ and } \quad \Vert S-T \Vert < \varepsilon.
$$
\end{dfn}

Dantas et al. \cite{DGMM18} initiated in 2018 the study of the Bishop–Phelps–Bollobás property for compact operators.

\begin{dfn}
    Let $X, Y$ be Banach spaces. We say that $(X,Y)$ has the \textit{Bishop-Phelps-Bollob\'as property for compact operators} (BPBp for compact operators for short) if, given $\varepsilon\in (0,1)$, there exists $\eta(\varepsilon)>0$ 
    such that for all $T\in S_{\mathcal{K}(X,Y)}$ and $x_0\in S_X$ satisfying $\Vert  T x_0 \Vert > 1- \eta(\varepsilon),$ there exist an operator $S\in S_{\mathcal{K}(X,Y)}$ and a point $u_0\in S_X$ such that
$$
\| Su_0 \| = 1, \quad \Vert u_0 - x_0 \Vert < \varepsilon, \quad \text{ and } \quad \Vert S-T \Vert < \varepsilon.
$$
\end{dfn}

We refer the reader to \cite{DGMO} and \cite{DGMM18} for further developments on
the BPBp and the BPBp for compact operators, respectively.

In 2015, Kim et al. \cite{KLM15} introduced the notion of the generalized approximate hyperplane series property for a pair $(X,Y)$ of Banach spaces. This property characterizes when the pair $(\ell_1(X), Y)$ has the BPBp, where $\ell_1(X)$ denotes the Banach space of all $(x_k)$ in $X$ such that $\sum_{k=1}^\infty\|x_k\|$ converges.

\begin{dfn}\label{gahsp}
Let $X, Y$ be Banach spaces. We say that $(X,Y)$ has the \textit{generalized AHSP} if, for every $\varepsilon > 0$, there exists $0 < \eta(\varepsilon) < \varepsilon$ such that given sequences $(T_k) \subset S_{\mathcal{L}(X,Y)}$ and $(x_k) \subset S_X$, and a convex series $\sum_{k=1}^\infty \alpha_k$ such that
$$\left\| \sum_{k=1}^\infty \alpha_k T_k x_k \right\| > 1 - \eta(\varepsilon),$$
there exist a subset $A \subset \conjN$, $y^* \in S_{Y^*}$, and sequences $(S_k) \subset S_{\mathcal{L}(X,Y)}$, $(z_k) \subset S_X$ satisfying the following:
\begin{enumerate}
    \item[(1)] $\sum_{k \in A} \alpha_k > 1 - \varepsilon$,
    \item[(2)] $\|z_k - x_k\| < \varepsilon$ and $\|S_k - T_k\| < \varepsilon$ for all $k \in A$,
    \item[(3)] $y^*(S_k z_k) = 1$ for every $k \in A$.
\end{enumerate}
\end{dfn}

The authors in \cite{KLM15} provided several examples of pairs $(X,Y)$ of Banach spaces that satisfy the generalized AHSP. 
They also showed that it is enough to verify the conditions of the generalized AHSP for every finite convex series (of arbitrary length), using the same function $\varepsilon \mapsto \eta(\varepsilon)$ for all lengths. Furthermore, we may restrict our attention to sequences $(T_k) \subset B_{\mathcal{L}(X,Y)}$ and $(x_k) \subset B_X$. Later, in 2022, Kim and Lee \cite{KL22} studied this property for $Y$-valued function spaces over a base algebra $A$.

For a compact Hausdorff space $K$, we denote by $\mathcal{C}(K, X)$ the space of all continuous functions from $K$ into $X$. For a locally compact Hausdorff space $L$, we write $\mathcal{C}_0(L, X)$ for the Banach space of all continuous functions from $L$ into $X$ that vanish at infinity. For a completely regular space $\Omega$, we denote by $\mathcal{C}_b(\Omega, X)$ the Banach space of all bounded continuous functions from $\Omega$ into $X$.

If $Z$ denotes one of these spaces, with underlying domain $\Gamma$, and if
$t\in\Gamma$, we write $\delta_t(f)=f(t)$ for every $f\in Z$.

In this paper, we show that the denseness of norm attaining operators (resp. compact operators), the BPBp for operators, and the BPBp for compact operators pass from $(X,\mathcal{C}_0(L,Y))$ to $(X,Y)$. Conversely, we prove that the denseness of norm attaining compact operators and the BPBp for compact operators pass from $(X,Y)$ to $(X,\mathcal{C}_0(L,Y))$. Although the corresponding converse stability results fail for general linear operators, we show that if $(X,Y)$ has the generalized AHSP and $X$ is finite-dimensional, then $(X,\mathcal{C}_0(L,Y))$ has the generalized AHSP. We also discuss the corresponding results for $\mathcal{C}(K,Y)$, where $K$ is compact Hausdorff, and for $\mathcal{C}_b(\Omega,Y)$, where $\Omega$ is completely regular Hausdorff, pointing out the cases that remain open.

The techniques employed in this paper are inspired by methods that have proved useful in related settings. More precisely, the proofs of Theorems \ref{t1} and \ref{t3} are inspired by \cite{ACKLM15}, while the proofs of Theorems \ref{t2}, \ref{t4}, and \ref{t5} rely on techniques from \cite{CGKM14,KL22}.

\section{Results}

We begin this section by studying the denseness of norm attaining operators and norm attaining compact operators.

\begin{theorem}\label{t1}
    Let $X$, $Y$ be Banach spaces and let $L$ be a locally compact Hausdorff space. 
    \begin{enumerate}
        \item[\rm (i)] If $\na(X, \mathcal{C}_0(L,Y))$ is dense in $\mathcal{L}(X, \mathcal{C}_0(L,Y))$, then $\na(X, Y)$ is dense in $\mathcal{L}(X, Y)$.
        \item[\rm (ii)] If $\na(X, \mathcal{C}_0(L,Y)) \cap \mathcal{K}(X, \mathcal{C}_0(L,Y))$ is dense in $\mathcal{K}(X, \mathcal{C}_0(L,Y))$, then $\na(X, Y) \cap \mathcal{K}(X, Y)$ is dense in $\mathcal{K}(X, Y)$.
    \end{enumerate}
\end{theorem}

\begin{proof}
    We prove (ii), since the proof of (i) is analogous.
    
    Let $T \in \mathcal{K}(X,Y)$ and let $0 < \varepsilon < \|T\|$. Fix $t_0 \in L$. Since $L$ is locally compact, there exists a compact set $K$ such that $t_0 \in \interior K \subset K$. Define $U = \interior K$. By hypothesis, $L$ is a completely regular space. Hence, there is a continuous function $\phi: L\rightarrow [0,1]$ such that $\phi(t_0)=1$ and $\phi(t)=0$ for every $t \in L\setminus U$.

    Consider the bounded linear operator $R: X \rightarrow \mathcal{C}_0(L, Y)$ defined by $Rx(t)=\phi(t)Tx$ for all $x\in X$ and $t\in L$. Thus, $R$ is compact and
    $$\|R\| = \sup_{x \in B_X} \sup_{t \in L} \left\| \phi(t)Tx \right\| = \sup_{t \in L} |\phi(t)| \cdot \sup_{x \in B_X} \| Tx \| = \|T\|.$$
    By the assumption, there exist $S \in \mathcal{K}(X,\mathcal{C}_0(L,Y))$ and $u_0 \in S_X$ satisfying
    $$\|Su_0\| = \|S\| \quad \textnormal{ and } \quad \|S - R\| < \frac{\varepsilon}{4}.$$
    In particular, there is $t_1\in L$ such that $\| S u_0(t_1) \| = \| Su_0 \| = \|S\|$. 
    Now, define the bounded linear operator $\widetilde {S}: X \rightarrow Y$ by $\widetilde{S}=\delta_{t_1} \circ S$. Hence, $\widetilde{S}$ is compact and 
    $$\|\widetilde{S}\| = \sup_{x \in B_X} \left\| Sx(t_1) \right\| =  \| S u_0(t_1) \| = \| \widetilde{S}u_0 \|.$$
    So, $\widetilde{S} \in \na(X, Y) \cap \mathcal{K}(X, Y)$. Besides that,
    \begin{align*}
        \|S\|-\left\|R u_0(t_1)\right\| &= \|S u_0(t_1)\|-\|R u_0(t_1)\|\\
        &\leq \|S u_0(t_1)- R u_0(t_1)\| \leq \left\| S-R\right\|<\frac{\varepsilon}{4}.
    \end{align*}
    This implies that 
    $$\|Ru_0(t_1)\| > \|S\|-\frac{\varepsilon}{4} 
    > \|R\|-\frac{\varepsilon}{2} = \|T\|-\frac{\varepsilon}{2}.$$
    Consequently, $\phi(t_1)Tu_0\neq 0$. The only possibility is that $t_1\in U$ and
    $$\|T\|-\frac{\varepsilon}{2} < \left\|Ru_0(t_1)\right\| = \left\|\phi(t_1)Tu_0\right\| \leq \phi(t_1)\|T\|.$$
    Finally, if $x\in B_X$ then
    \begin{align*}
    \left\|\widetilde{S}x-Tx\right\| &= \left\|(\delta_{t_1}\circ S)x-Tx\right\|\\
    &\leq \left\| Sx(t_1)-Rx(t_1)\right\|+\left\| Rx(t_1)-Tx\right\|\\
    &\leq \left\| S-R\right\|+\left\lvert \phi(t_1)-1\right\rvert\left\| Tx\right\|\\
    &\leq \left\| S-R\right\|+(1-\phi(t_1))\|T\| < \left\| S-R\right\|+\frac{\varepsilon}{2}.
    \end{align*}
    Therefore, $\left\| \widetilde{S}-T\right\| \leq \left\| S-R\right\| + \frac{\varepsilon}{2}<\varepsilon$, which completes our proof.
\end{proof}

\begin{remark}
Both assertions of the preceding theorem remain valid if $\mathcal{C}_0(L,Y)$ is replaced by $\mathcal{C}(K,Y)$, where $K$ is a compact Hausdorff space. Indeed, the proof can be simplified by considering, for each $T\in\mathcal{L}(X,Y)$, the operator $Rx(t)=Tx (x\in X,\ t\in K).$
In the complex case, both assertions also follow from \cite[Theorem 2.4]{KL22} by taking the unital function algebra $A =\mathcal{C}(K)$ and the $Y$-valued function space $\mathcal{A} = \mathcal{C}(K,Y)$.

We do not know whether the corresponding assertions remain valid when $\mathcal{C}_0(L,Y)$ is replaced by $\mathcal{C}_b(\Omega,Y)$, where $\Omega$ is a completely regular Hausdorff space.
\end{remark}

The converse of the first item of the above theorem is false. Indeed, if $L = K = [0,1]$ and $Y=\conjR$ or $\conjC$, then $\mathcal{C}_0(L,Y) = C(K, Y)$ does not have the Lindenstrauss property B, as shown in \cite[Theorem A]{S83}, even though $\na(X,Y)$ is dense in $\mathcal{L}(X,Y)$ for every Banach space $X$, by the Bishop-Phelps theorem. Next, we present the converse of the second item.

\begin{theorem}\label{t2}
  Let $X$, $Y$ be Banach spaces and let $L$ be a locally compact Hausdorff space. If the set $\na(X,Y)\cap\mathcal{K}(X,Y)$ is dense in $\mathcal{K}(X,Y)$, then $\na(X,\mathcal{C}_0(L,Y))\cap\mathcal{K}(X,\mathcal{C}_0(L,Y))$ is dense in $\mathcal{K}(X,\mathcal{C}_0(L,Y))$.
\end{theorem}

\begin{proof}
    Fix $0<\varepsilon<1$ and $T\in S_{\mathcal{K}(X, \mathcal{C}_0(L,Y))}$. Let $x_0\in S_X$ such that $\left\| Tx_0 \right\|>1-\frac{\varepsilon}{8}$. There is $t_0\in L$ with $\left\| Tx_0(t_0) \right\|>1-\frac{\varepsilon}{8}$. Let $T':X\rightarrow Y$ be 
    defined by $T'=\delta_{t_0}\circ T$. Notice that $T'$ is a compact operator and $1\geq \left\| T' \right\|>1-\frac{\varepsilon}{8}$. 

    Since $\na(X,Y)\cap\mathcal{K}(X,Y)$ is dense in $\mathcal{K}(X,Y)$, there are $S'\in \mathcal{K}(X,Y)$ and $u_0\in S_X$ such that $\left\| S'u_0 \right\|=\left\| S' \right\|$ and $\left\| T'-S' \right\|<\frac{\varepsilon}{8}$. Hence, $1-\frac{\varepsilon}{4}<\left\| S' \right\|<1+\frac{\varepsilon}{8}$.

    We consider $\delta=\frac{\varepsilon}{8}$. By the assumption, $T$ is a compact operator, then $\overline{T(B_X)}$ is a compact set of $\mathcal{C}_0(L,Y)$. This implies $\overline{T(B_X)}\subset \bigcup^{m}_{j=1}B(T\widetilde{x}_j, \delta)$ for some $\widetilde{x}_1, \dots , \widetilde{x}_m\in B_X$, where $B(T\widetilde{x}_j, \delta)$ denotes the open ball with center $T\widetilde{x}_j$ and radius $\delta$. We define the set 
    \begin{equation*}
        U=\bigcap^{m}_{j=1}\left\{t\in L:\, \left\| T\widetilde{x}_j(t)-T'\widetilde{x}_j\right\|<\frac{\varepsilon}{4} \right\}.
    \end{equation*}
    Observe that $U$ is open because it is a finite intersection of open subsets of $L$ and $t_0\in U$. Thus, there exists a compact set $K$ such that $t_0\in \interior K \subset K\subset U$. We consider $V=\interior K$. There is a continuous function $\phi: L \rightarrow [0,1]$ such that $\phi(t_0)=1$ and $\phi(t)=0$ if $t\in L\setminus V$.

    Let $S: X\rightarrow \mathcal{C}_0(L,Y)$ be the bounded linear operator defined by
    \begin{equation*}
        Sx(t)=\phi(t)\frac{S'x}{\left\| S' \right\|}+(1-\phi(t))Tx(t)
    \end{equation*}
    for all $x\in X$ and $t\in L$.
    We notice that $S$ is a compact operator and, for every $x\in B_X$ and $t\in L$,
    \begin{equation*}
    \left\| Sx(t)\right\| \leq |\phi(t)|+|1-\phi(t)|=1 \, \, \text{and}\,\, \left\| Su_0(t_0)\right\|=1. 
    \end{equation*}
    Hence, $\left\| Su_0\right\|=\left\| S\right\|=1$. This implies $S\in \na(X,\mathcal{C}_0(L,Y))\cap\mathcal{K}(X,\mathcal{C}_0(L,Y))$. 
    
    Given $x\in B_X$ and $t \in V$, we have $\|Tx - T\widetilde{x}_j\| < \delta$ for some $j \in \{1, \dots, m\}$ and $\left\| T\widetilde{x}_j(t)-T'\widetilde{x}_j\right\|<\frac{\varepsilon}{4}$. Then, 
    \begin{align*}
        \left\|Sx(t)-Tx(t)\right\|&= \left\|\phi(t)\left(\frac{S'x}{\left\|S'\right\|}-Tx(t) \right)\right\|\\
        &\hspace{-1cm}\leq \left\|\frac{S'x}{\left\|S'\right\|}-S'x\right\|+\left\|S'x-Tx(t)\right\|\\
        &\hspace{-1cm}\leq |1-\left\|S'\right\||+\left\|S'x-T'x\right\|+\left\|T'x-Tx(t)\right\|\\
        &\hspace{-1cm}\leq \frac{\varepsilon}{4}+\frac{\varepsilon}{8}+\left\|T'x-T'\widetilde{x}_j\right\|+\left\|T'\widetilde{x}_j-T\widetilde{x}_j(t)\right\| +\left\|T\widetilde{x}_j(t)-Tx(t)\right\|\\
        &\hspace{-1cm}\leq \frac{\varepsilon}{4}+\frac{\varepsilon}{8}+2\delta+\frac{\varepsilon}{4}<\varepsilon.
    \end{align*}
    In the case that $x\in B_X$ and $t\in L\setminus V$, 
    $$\left\|Sx(t)-Tx(t)\right\|= \left\|\phi(t)\left(\frac{S'x}{\left\|S'\right\|}-Tx(t) \right)\right\|=0< \varepsilon.$$ Therefore, $\left\|S-T\right\|<\varepsilon$ and the proof is completed.
\end{proof}

\begin{remark}
An analogous result holds when $\mathcal{C}_0(L,Y)$ is replaced either by $\mathcal{C}(K,Y)$, where $K$ is a compact Hausdorff space, or by $\mathcal{C}_b(\Omega,Y)$, where $\Omega$ is a completely regular Hausdorff space. Indeed, in both cases, after obtaining an open neighbourhood $U$ of $t_0$ in the underlying space, complete regularity gives a continuous function $\phi$ with values in $[0,1]$ such that $\phi(t_0)=1$ and $\phi(t)=0$ for every $t \notin U$.
In the complex case, the assertion for $\mathcal{C}(K,Y)$ also follows from \cite[Corollary 2.3]{KL22} by taking the unital function algebra $A = \mathcal{C}(K)$ and the $Y$-valued function space $\mathcal{A} = \mathcal{C}(K,Y)$.
\end{remark}

In what follows, we generalize Theorems \ref{t1} and \ref{t2} to the BPBp and the BPBp for compact operators.

\begin{theorem}\label{t3}
    Let $X$, $Y$ be Banach spaces and let $L$ be a locally compact Hausdorff space. 
    \begin{enumerate}
        \item[\rm (i)] If the pair $(X, \mathcal{C}_0(L,Y))$ has the BPBp, then the pair $(X,Y)$ has the BPBp.
        \item[\rm (ii)] If the pair $(X, \mathcal{C}_0(L,Y))$ has the BPBp for compact operators, then the pair $(X,Y)$ has the BPBp for compact operators.
    \end{enumerate}
\end{theorem}

\begin{proof}
    Again, we prove (ii), since the proof of (i) is analogous.
    
    Suppose that $(X,\mathcal{C}_0(L,Y))$ has the BPBp for compact operators with a function $\varepsilon \mapsto \eta(\varepsilon)$. Fix $\varepsilon>0$ and consider $T \in S_{\mathcal{K}(X,Y)}$ and $x_0 \in S_X$ such that $\left\| T x_0 \right\| > 1- \eta\left(\frac{\varepsilon}{4}\right).$

    As in Theorem \ref{t1}, take $t_0 \in L$, a compact set $K$ such that $t_0 \in U = \interior K$, and a continuous function $\phi: L\rightarrow [0,1]$ such that $\phi(t_0)=1$ and $\phi(t)=0$ for every $t \in L\setminus U$.

    Considering the compact operator $R: X \rightarrow \mathcal{C}_0(L, Y)$ defined by $Rx(t)=\phi(t)Tx$ for all $x\in X$ and $t\in L$, we can show that $\|R\| = \|T\| = 1$, as in Theorem \ref{t1}. Notice that 
    $$\left\| R x_0(t) \right\| = |\phi(t)| \| T x_0 \| \leq \| T x_0 \|.$$ Since $ \left\| R x_0(t_0) \right\| = \left\| T x_0 \right\|$, we have
    $$\left\| R x_0 \right\| = \left\| T x_0 \right\| > 
    1-\eta\left(\frac{\varepsilon}{4}\right).$$
    By the assumption, there exist $S \in S_{\mathcal{K}(X,\mathcal{C}_0(L,Y))}$ and $u_0 \in S_X$ satisfying
    $$\|Su_0\| = 1, \quad \|u_0 - x_0\| < \frac{\varepsilon}{4}, \quad \textnormal{ and } \quad \|S - R\| < \frac{\varepsilon}{4}.$$
    Let $t_1\in L$ be such that $\| S u_0(t_1) \| = \| Su_0 \| = 1$ and define the compact operator $\widetilde {S}: X \rightarrow Y$ by $\widetilde{S}=\delta_{t_1} \circ S$. Thus, analogous arguments to those presented in Theorem \ref{t1} show that $$\|\widetilde{S}\| = \|\widetilde{S}u_0\| \quad \textnormal{ and } \quad \|\widetilde{S} - T\| < \varepsilon.$$
    This completes the proof.
\end{proof}

\begin{remark}
The above result remains true if $\mathcal{C}_0(L,Y)$ is replaced by $\mathcal{C}(K,Y)$, where $K$ is a compact Hausdorff space. The proof is almost the same, using the constant operator $Rx(t)=Tx$. In the complex case, this also follows from \cite[Theorem 2.7]{KL22} with $A = \mathcal{C}(K)$ and $\mathcal{A} = \mathcal{C}(K,Y)$.

We do not know whether the analogous result holds for $\mathcal{C}_b(\Omega,Y)$, where $\Omega$ completely regular Hausdorff space.
\end{remark}

\begin{theorem}\label{t4}
    Let $X$, $Y$ be Banach spaces and let $L$ be a locally compact Hausdorff space. If the pair $(X,Y)$ has the BPBp for compact operators, then the pair $(X, \mathcal{C}_0(L,Y))$ has the BPBp for compact operators.
\end{theorem}

\begin{proof}
    Suppose that $(X,Y)$ has the BPBp for compact operators with a function $\varepsilon \mapsto \eta(\varepsilon)$. Fix $\varepsilon>0$ and consider $\lambda>0$ such that $\lambda<\frac{\varepsilon}{4}$ and $\eta(\lambda)<\frac{\varepsilon}{4}$.

    Let $T\in S_{\mathcal{K}(X, \mathcal{C}_0(L,Y))}$, let $x_0\in S_X$ and let $t_0\in L$ such that $\left\| T x_0 (t_0) \right\| > 1-\eta(\lambda).$

    We define the bounded linear operator $T':X\rightarrow Y$ by $T'=\frac{\delta_{t_0}\circ T}{\left\| \delta_{t_0}\circ T\right\|}$. Notice that $T'\in S_{\mathcal{K}(X,Y)}$, $\left\|  T'x_0\right\|> 1-\eta(\lambda)$, and $\left\| \delta_{t_0}\circ T\right\|>1-\eta(\lambda)$. Since $(X,Y)$ has the BPBp for compact operators, there are $S'\in S_{\mathcal{K}(X,Y)}$ and $u_0 \in S_X$ such that
    \begin{equation*}
\left\| S'u_0\right\|=1, \quad \left\| x_0-u_0\right\|<\lambda, \quad \text{ and } \quad \left\|T'-S'\right\|<\lambda.
    \end{equation*}

    We consider $\delta=\frac{\varepsilon}{4} \left( 1-\eta(\lambda) \right)$. By hypothesis, there are $\widetilde{x}_1,\dots,\widetilde{x}_m\in B_X$ such that $\overline{T(B_X)}\subset\bigcup^{m}_{j=1}B(T\widetilde{x}_j, \delta)$. We define the open subset 
    \begin{equation*}
    U=\bigcap^{m}_{j=1}\left\{t\in L:\, \left\| T\widetilde{x}_j(t)-T'\widetilde{x}_j\right\|<\eta (\lambda) \right\},
    \end{equation*}
    that contains $t_0$, and a continuous function $\phi : L \rightarrow [0,1]$, as in Theorem \ref{t2}. Then, we consider the compact operator $S: X\rightarrow \mathcal{C}_0(L,Y)$ defined by
    \begin{equation*}
        Sx(t)=\phi(t)S'x+(1-\phi(t))Tx(t)
    \end{equation*}
    for all $x\in X$ and $t\in L$. Using the same arguments as in Theorem \ref{t2}, we conclude that $\left\| Su_0\right\|=1$ and $\left\| S-T\right\|<\varepsilon$.
\end{proof}

\begin{remark}
Theorem \ref{t4} holds if $\mathcal{C}_0(L,Y)$ is replaced either by $\mathcal{C}(K,Y)$, where $K$ is a compact Hausdorff space, or by $\mathcal{C}_b(\Omega,Y)$, where $\Omega$ is a completely regular Hausdorff space. Indeed, the proof can be adapted in both cases using the complete regularity to obtain a continuous function $\phi$ with values in $[0,1]$ such that
$\phi(t_0)=1$ and $\phi(t)=0$ for every $t\notin U$.
In the complex case, the assertion for $\mathcal{C}(K,Y)$ also follows from \cite[Corollary 2.6]{KL22} by taking the unital function algebra $A = \mathcal{C}(K)$ and the $Y$-valued function space $\mathcal{A} = \mathcal{C}(K,Y)$.
\end{remark}

As in the case of Theorem \ref{t1}, the converse of the first item of Theorem \ref{t3} is false, by the same argument presented after Theorem \ref{t1}.

\vspace{0.2cm}

In view of \cite[Corollary 2.10]{KL22}, it is natural to ask whether the converse of Theorem \ref{t4} holds for the generalized AHSP under the additional assumption $\mathcal{L}(X,\mathcal{C}_0(L,Y)) = \mathcal{K}(X,\mathcal{C}_0(L,Y))$. However, this condition is quite restrictive, since, as we show below, it implies that either $L$ is finite or $X$ is finite-dimensional.

\begin{prop}
   Let $X$, $Y$ be Banach spaces with $Y\neq\{0\}$, and let $L$ be a locally compact Hausdorff space. If $\mathcal{L}(X,\mathcal{C}_0(L,Y)) = \mathcal{K}(X,\mathcal{C}_0(L,Y))$ then $L$ is finite or $X$ is finite-dimensional.  
\end{prop}

\begin{proof}
Suppose that $L$ is infinite and that $X$ is infinite-dimensional. Since $L$ is a locally compact Hausdorff space, there exists a sequence $(U_n)_{n=1}^\infty$ of pairwise disjoint non-empty open subsets of $L$ with compact closures. Hence, for each $n\in\mathbb N$, we may choose a function $\phi_n\in \mathcal{C}_0(L)$ such that $\text{supp}(\phi_n) \subset U_n$ and $\left\| \phi_n \right\|=1$. Fix $y_0\in S_Y$. Since the supports of the functions $\phi_n$ are pairwise disjoint, the operator $J: c_0 \to \mathcal{C}_0(L,Y)$ defined by
$J((a_n)_{n=1}^\infty)(t) = \sum_{n=1}^\infty a_n \phi_n(t)y_0$
for $(a_n)_{n=1}^\infty\in c_0$ and $\ t\in L$, is well-defined, linear, and satisfies $\left\|J((a_n)^\infty_{n=1})\right\|=\left\|(a_n)^\infty_{n=1}\right\|$.
Thus, $J$ is an isometric embedding, ensuring that $\mathcal{C}_0(L,Y)$ contains a closed subspace isometrically isomorphic to $c_0$.

On the other hand, since $X$ is infinite-dimensional, by the Josefson-Nissenzweig theorem \cite[p. 219]{D84}, there is a weak$^*$-null sequence $(x_n^*)_{n=1}^\infty \subset S_{X^*}$. Define $T:X\to c_0$ by $Tx=(x_n^*(x))_{n=1}^\infty$ for $x\in X$. Then $T$ is a bounded linear operator. Moreover, $T$ is not compact (see \cite[p. 223]{D84}).

Therefore, $J\circ T:X\to \mathcal{C}_0(L,Y)$ is bounded, linear, and noncompact. Hence, $\mathcal{L}(X,\mathcal{C}_0(L,Y)) \neq \mathcal{K}(X,\mathcal{C}_0(L,Y))$, which proves the result.
\end{proof}

\begin{remark}
The above result remains valid when $\mathcal{C}_0(L,Y)$ is replaced by either $\mathcal{C}(K,Y)$ or $\mathcal{C}_b(\Omega, Y)$, where $K$ is a compact Hausdorff space and $\Omega$ is a completely regular Hausdorff space.
\end{remark}

Therefore, by the previous proposition, the only cases to be considered are those in which $L$ is finite or $X$ is finite-dimensional. We shall prove below that the converse does hold for the generalized AHSP when $X$ is finite-dimensional. We will not consider separately the case where $L$ is finite, because, if $L=\{t_1,\ldots,t_n\}$, then $\mathcal{C}_0(L,Y)$ is isometrically isomorphic to $\left(\bigoplus_{k=1}^n Y\right)_\infty$, and the desired conclusion follows from \cite[Proposition 2.4]{ACKLM15}, together with the equivalence between the generalized AHSP for $(X,Y)$ and the BPBp for $(\ell_1(X),Y)$.

\begin{theorem}\label{t5}
    Let $X$, $Y$ be Banach spaces and let $L$ be a locally compact Hausdorff space. If the pair $(X,Y)$ has the generalized AHSP and $X$ is finite-dimensional, then the pair $(X, \mathcal{C}_0(L,Y))$ has the generalized AHSP. Equivalently, if the pair $(\ell_1(X),Y)$ has the BPBp and $X$ is finite-dimensional, then the pair $(\ell_1(X), \mathcal{C}_0(L,Y))$ has the BPBp.
\end{theorem}

\begin{proof}
 Suppose that $(X,Y)$ has the generalized AHSP with a function $\varepsilon \mapsto \eta(\varepsilon)$. Take $\varepsilon>0$, $(T_k)^{n}_{k=1}\subset B_{\mathcal{L}(X, \mathcal{C}_0(L,Y))}$, $(x_k)^{n}_{k=1}\subset B_{X}$, and a convex sum $\sum^{n}_{k=1}\alpha_k$ such that
    \begin{equation*}
        \left\| \sum_{k=1}^{n} \alpha_k T_k x_k \right\| >1-\eta\left(\frac{\varepsilon}{4}\right).
    \end{equation*}
Then, there is $t_0\in L$ such that
 \begin{equation}\label{i1}
        \left\| \sum_{k=1}^{n} \alpha_k T_k x_k(t_0) \right\|>1-\eta\left(\frac{\varepsilon}{4}\right).
    \end{equation}
    
For every $k\in \{1, \ldots, n \}$, we consider the bounded linear operator $R_k:X\rightarrow Y$ defined by $R_k=\delta_{t_0}\circ T_k$. Notice that $ \left\| R_k\right\|\leq 1$ for all $k\in \{1, \ldots, n \}$. By (\ref{i1}), 
    \begin{equation*}
         \left\| \sum_{k=1}^{n} \alpha_k R_k x_k \right\|=\left\| \sum_{k=1}^{n} \alpha_k  (\delta_{t_0}\circ T_k)x_k \right\|=\left\| \sum_{k=1}^{n} \alpha_k   T_kx_k(t_0) \right\|>1-\eta\left(\frac{\varepsilon}{4}\right).
    \end{equation*}
By the assumption, there exist a subset $A\subset \{1, \ldots, n \}$, $y^*\in S_{Y^*}$, and sequences $(S_k)^{n}_{k=1}\subset S_{\mathcal{L}(X,Y)}$ and $(z_k)^{n}_{k=1}\subset S_X$ satisfying the following:
    \begin{enumerate}
    \item[(1)] $\sum_{k \in A} \alpha_k > 1 - \frac{\varepsilon}{4}$,
    \item[(2)] $\|z_k - x_k\| < \frac{\varepsilon}{4}$ and $\|S_k - R_k\| < \frac{\varepsilon}{4}$ for all $k \in A$,
    \item[(3)] $y^*(S_k z_k) = 1$ for every $k \in A$.
    \end{enumerate}
We choose $\beta>0$ such that $\left\| R_k-S_k\right\|<\beta<\frac{\varepsilon}{4}$ for all $k\in A$ and take $\delta=\frac{\varepsilon}{4}-\beta$. Since $X$ is finite-dimensional, $\bigcup_{k\in A}\overline{T_k(B_X)}$ is a compact set. Then, for each $k\in A$, there exist $\widetilde{x}^k_1, \ldots, \widetilde{x}^k_{m_k} \in B_X$ such that $\bigcup_{k\in A}\overline{T_k(B_X)}\subset \bigcup_{k\in A}\bigcup^{m_k}_{j=1}B(T_k\widetilde{x}^k_j, \delta)$. Consider the set
\begin{equation*}
    U=\bigcap_{k\in A}\bigcap^{m_k}_{j=1}\left\{t\in L:\, \left\| T_k\widetilde{x}^k_j(t)-S_k\widetilde{x}^k_j\right\|<\beta \right\}.    
\end{equation*}
Observe that $U$ is open. Furthermore, for every $k\in A$ and $j\in \{1,\ldots, m_k\}$, we see that 
$$\left\| T_k\widetilde{x}^k_j(t_0)-S_k\widetilde{x}^k_j\right\|=\left\| R_k\widetilde{x}^k_j-S_k\widetilde{x}^k_j\right\|\leq \left\| R_k-S_k\right\|<\beta.$$
Which implies that $t_0\in U$. Thus, as in Theorem \ref{t2}, there is a compact set $K$ such that $t_0 \in V = \interior K \subset K \subset U$ and a continuous function $\phi: L \rightarrow [0,1]$ such that $\phi(t_0)=1$ and $\phi(t)=0$ for every $t \in L\setminus V$.

For every $k\in A$, we define the operator $\widetilde {S}_k: X\rightarrow \mathcal{C}_0(L,Y)$ by $$\widetilde {S}_kx=\phi S_kx+(1-\phi)T_kx.$$
We claim that $\left\| \widetilde {S}_k \right\|=1$ for all $k\in A$. Indeed, it is easy to see that $\left\| \widetilde {S}_k \right\| \leq 1$. Let $\psi:\mathcal{C}_0(L,Y)\rightarrow \mathbb{K}$ be the functional defined by $\psi=y^*\circ \delta_{t_0}$. Notice that $\|\psi\| \leq 1$. 
Furthermore, 
$$\psi(\widetilde{S}_kz_k) =y^*(\widetilde{S}_kz_k(t_0)) 
= y^*(\phi(t_0)S_kz_k+(1-\phi(t_0))T_kz_k(t_0)) = y^*(S_kz_k)=1$$
for all $k\in A$. Therefore, $\left\| \psi\right\|=1$. So, we obtain that $\left\| \widetilde{S}_k\right\|=1$ for all $k\in A$ and also that the third condition of the Definition \ref{gahsp} is satisfied.

To finish the proof, let $x\in B_X$. Hence, for each $k\in A$, there is $j_k\in \{1, \ldots, m_k \}$ such that $\left\| T_kx-T_k\widetilde{x}^k_{j_k}\right\|<\delta$. If $t\in V$, then $\left\| T_k\widetilde{x}_j^k(t)-S_k\widetilde{x}_j^k\right\|<\beta$ for all $k\in A$ and for all $j\in \{1, \ldots, m_k \}$. This implies that
 \begin{align*}
     \left\| T_kx(t)-\widetilde{S}_kx(t)\right\| &= 
     \left\| \phi(t)(T_kx(t)-S_kx)\right\| \leq \left\| T_kx(t)-S_kx\right\|\\
     &\leq \left\| T_kx(t)-T_k\widetilde{x}^k_{j_k}(t)\right\|+\left\| T_k\widetilde{x}^k_{j_k}(t)-S_k\widetilde{x}^k_{j_k}\right\|\\
     &\hspace{0.4cm}+\left\| S_k\widetilde{x}^k_{j_k}-R_k\widetilde{x}^k_{j_k}\right\|+\left\| R_k\widetilde{x}^k_{j_k}-R_kx\right\|+\left\| R_kx-S_kx\right\|\\
     &< 2\delta+3\beta = 2\left(\frac{\varepsilon}{4}-\beta \right)+3\beta=\frac{\varepsilon}{2}+\beta<\frac{3\varepsilon}{4}.
 \end{align*}
On the other hand, if $t\in L\setminus V$, then $\phi(t)=0$. So, $\left\| T_kx(t)-\widetilde{S}_kx(t)\right\| =0$. Thus, $\left\| T_k-\widetilde{S}_k\right\|\leq \frac{3\varepsilon}{4}<\varepsilon$ for all $k\in A$.
\end{proof}

\begin{remark}
The same conclusion holds when $\mathcal{C}_0(L,Y)$ is replaced either by $\mathcal{C}(K,Y)$, where $K$ is a compact Hausdorff space, or by $\mathcal{C}_b(\Omega,Y)$, where $\Omega$ is a completely regular Hausdorff space. Indeed, the proof can be adapted in both cases using complete regularity to obtain the required function $\phi$.
In the complex case, the assertion for $\mathcal{C}(K,Y)$ also follows from \cite[Corollary 2.10]{KL22}, since $X$ finite-dimensional implies that $\mathcal{L}(X,\mathcal{C}(K,Y)) = \mathcal{K}(X,\mathcal{C}(K,Y))$.
\end{remark}


\end{document}